%% file: pappairw-arXiv-copy-2.tex
\documentclass[18pt]{article}

\usepackage[colorlinks,citecolor=blue,urlcolor=blue]{hyperref}

\input deff.tex

\renewcommand{\baselinestretch}{1.5}

\usepackage[authoryear,semicolon]{natbib}
\begin{document}
\noindent  {\footnotesize This is the peer reviewed version of the following article: STAT, \textbf{5}, (2016), 286-294, which has been published in final form at doi:10.1002/sta4.122. This article may be used for non-commercial purposes in accordance with Wiley Terms and Conditions for Self-Archiving;
also on Arxiv 1512.09016}.\\

\noindent{{\begin{center}{\Large \bf  Pairwise Markov properties for regression graphs}\end{center}}
\noindent {\begin{center}{\bf{Kayvan Sadeghi}}\\
{\em University of Cambridge, Cambridge, UK}\end{center}

\noindent {\begin{center}{\bf{Nanny Wermuth}} \\
 {\em Chalmers University of Technology, Gothenburg, Sweden and Gutenberg-University, Mainz, Germany}\end{center}

\noindent {\bf Abstract:} {\small  With a  sequence of regressions,  one may generate joint probability distributions. One starts with
a joint, marginal distribution of context variables having possibly a  concentration graph structure  and continues  with an ordered sequence  of   conditional distributions, named regressions in joint responses.  The involved random variables may be discrete, continuous or of both types. Such a generating process specifies  for each response a conditioning set which contains just its regressor variables and  it leads to at least one valid ordering of all nodes in  the corresponding regression graph which has  three types of edge; one for undirected dependences among context variables, another  for undirected dependences among joint responses and one for any directed dependence of a response on a regressor variable.
For this regression graph, there are several definitions of pairwise  Markov properties, where each interprets the conditional independence associated with a missing edge in the graph in a different way. We explain how these  properties
arise, prove their equivalence for compositional graphoids and  point at the equivalence of each one of them  to the global Markov property.
}

\section{Introduction} Regression graph models are the subclass of graphical Markov models that is best suited to capture
pathways of development, no matter whether the generating process arises  in observational studies or  in intervention studies, see \citet{Wer15}. For prospective studies, time provides a  partial ordering of a finite number of  variables and leads to  an ordered sequence of joint and single response variables. In other studies,  any  ordering is provisional but, typically,  may be agreed upon by  researchers in the given field of study.

Regression graph models  extend models proposed  by \citet{CoxWer93}, illustrated there with several small sets of data having some  joint responses and  just linear relations; for linear regressions see e.g.~\citet{Weisb14}. Additional in regression graph models are context variables which capture  baseline features  of  study individuals or of study conditions. Context  variables may be controlled  or fixed by study design. Regression graphs represent an essential extension of fully directed acyclic graphs
whenever one wants to model  that explanatory variables  affect  directly several response variables at the same time and to take the baseline information of context variables explicitly  into account.

Individual random variables  can be continuous, discrete or of both types.  For such mixed responses, the conditional distributions may, for instance, be conditional Gaussian
regressions; see \citet{LauWer89},  be approximated by generalized linear regressions;  see \citet{McCulNel89}, \citet{AndSkov10}, or if needed, be any other type of  nonlinear regression, which permits to represent  the relevant conditional independences and dependences.

 For discrete  variables,  ordered sequences of joint regressions  without context variables have been studied as  chain graph  models of type IV by  \citet{Drton09}.
A parameterization and maximum-likelihood estimates were  derived by \citet{MarLup11}.

For context variables in general, the independence structure is captured by  an undirected graph in vertex set $V=\{1, \ldots, d_V\}$ and edge set $E$, known as the concentration graph. The name reminds one that for regular joint Gaussian
distributions,  a missing edge for vertex pair $i,j$ shows as  an $ij$-zero in its concentration matrix, the inverse covariance matrix; see \citet{Dempster72}, \citet{CoxWer93}. In general, at most one  undirected $ij$-edge, $i\ful j$, couples the vertices $i,j$, and each missing edge for $i,j$ means  $X_i$ is independent of $X_j$ given the remaining variables. This is written compactly in terms of vertices, also called  nodes,
 as $i\ci j|V\setminus\{i,j\}$; see \citet{Dawid79}.

For exclusively discrete context variables, the
joint distributions have been studied as Markov fields and log-linear interaction models; see \citet{DarLauSp80}, named later also discrete concentration graph models.  For these, one also knows when single unobserved variables can  be identified; see \citet{StangVant13}.

 For regression graphs in node set $N=\{1, \ldots d\},$  where 
  each node represents a variable,  the edge set contains in general  three types of edge of different interpretation. The  $ij$-edge among two context nodes is denoted by a full line, $i\ful j$.
Edges  for directed dependences are $ij$-arrows, $i \fla j$. Each $ij$-arrow starts at a regressor node  $j$ and points to a response $i$.  
The $ij$-edge between two individual
response nodes within a given joint response is denoted by a dashed line, $i \dal j$. In some  literature, this is  replaced by an double-headed arrow. 
Undirected graphs of this type are  named conditional covariance graphs to remind one that for  joint Gaussian distributions, a missing $ij$-edge shows as an $ij$-zero  in the conditional covariance matrix of a  joint response  $X\!:=(X_1, \ldots X_p)$, given the union of their individual  regressors.

Among the results  available so far for regression graphs are (1) a simple graphical criterion to decide whether two regression graphs, with the same node and edge set but different types of edge, define the same independence structure; see \citet{WerSadeg12}, (2) path criteria and zero-matrix criteria for separation;  see
\citet{SadegLau14}, \citet{Wer15}, (3) induced graphs, after marginalising over some nodes and conditioning on others, which preserve the independence structure
of the generating  regression graph, see \citet{Sadeg13}, (4) induced graphs, which only predict dependences  generated by a regression graph between  $X_\alpha$ and $X_\beta$ given $X_c$,  for disjoint  $\alpha, \beta, c, m$ partitioning $N$.  This requires an edge-minimal graph, that is one  in which no edge can be removed without introducing another independence constraint; see \citet{Wer12}, (5) implementations  in the computing environment R; see \citet{SadegMar12}.

Independence properties implied by the standard graph theoretic separation criterion on graphs listed here as $(i)$ to $(v)$ in Section 3, were introduced  by \citet{PearlPaz87} for concentration graphs and said to define a graphoid.   More importantly, they  derived that for graphoids all independences  implied by the given graph, may  be obtained either from the list of independences defined by its missing edges or by separation in undirected graphs.
This became known later as the equivalence of the pairwise Markov property and the global Markov property.

Such an equivalence has been proven for more general classes of graphs, including regression graphs, for which the global  Markov property is also generalized
compared to concentration graphs; see  \citet{SadegLau14}. This uses  the  additional independence property of composition, listed here as property $(vi)$ in Section 4. For such compositional graphoids, several pairwise Markov properties may be defined, which  give
alternative  independence interpretations to a missing edge. One question to be solved was whether each of them is suitable to  derive the regression-graph structure that is all independences implied
by the graph.

Here, we define four  pairwise Markov properties and  show with Theorem \ref{thm:1} that they are equivalent with respect to  regression graphs.  Thus, any one of the four pairwise properties characterizes the regression-graph structure. By using the  equivalence of one of them to the global Markov property, proven in \citet{SadegLau14}, it follows that each one is also equivalent to the global Markov property.

The smallest conditioning set for the independence of two response components will be shown to be
the set of important regressors for just one of the responses. This result  for independences is of special importance since it  contrasts with a well-known result for estimating the
dependences of responses on their regressors:  when strong residual dependences remain among the responses, after one has been regressing each response
component separately on  different sets of regressors, distorted parameter estimates may be obtained, see \citet{Haavelm43}, \citet{Zellner62},
\citet{DrtonRich04}. For an exposition of these ideas in terms of Gaussian variables, standardized to have zero means and unit variances, see
\citet{CoxWer93}, beginning of Section 4.

We introduce in Section 2 more definitions and concepts for  regression graphs.
In Section 3, we provide independence properties, give four pairwise Markov properties for regression graphs, and prove their equivalence for compositional graphoids.
We end with a discussion in Section 4.
\section{Some more definitions and concepts for regression graphs}
A simple graph with a finite node set $N=\{1,\dots, d\}$ has in its edge set $E$  at most one $ij$-edge for nodes $i,j$ in $N$.
When $i,j$ are the endpoints of an edge, the node pair is said to be coupled or to be adjacent and  denoted  by $i\sim j$.
A simple graph is complete when all pairs are coupled by an edge.

A parent of node $i$ is the starting node $j$ of an arrow pointing to $i$ and indicates a regressor variable $X_j$ of   response  $X_i$. The set of parents of a node $i$ is denoted by $\pa(i)$.
An $ij$-path is a sequence of edges passing through distinct inner nodes, connecting the endpoints $i$ and $j$ of the path. By convention, an $ij$-edge is the shortest  type of path.

 A subgraph induced by a subset $a$ of $N$ consists of
  the node set $a$ and of the edges
  present in the graph among the nodes of $a$.
 A subgraph is connected if there is a path between any  distinct node pair. 
 A maximally connected subgraph is called a connected component.

 A direction-preserving $ij$-path to $i$  from $j$  consists of arrows, all  pointing from $j$  towards $i$. In  such a direction-preserving path,  node $j$ is named an ancestor of $i$.
 An anterior $ij$-path consists of  a direction-preserving path  followed by an undirected  path, such as:
 \begin{displaymath}
i \
\fla \underbrace{\overbrace{ \snode \fla \snode, \ldots,  \snode \fla  \snode}^{\text {\normalsize{ancestors of  \textit{i}}}}\ful  \snode\ful \snode, \ldots , \snode \ful j}_{\text{\normalsize{anteriors of \textit{i}}}}.
 \end{displaymath}
Thus, anterior nodes extend the notion of ancestors. The set of anterior nodes of node $i$ is denoted by $\ant(i)$.

The regression graph, \Greg\!,  is a simple graph with a finite node set $N$\! and an edge set of  three different types of edge. Here we use arrows,  dashed and full lines.
The properties of \Greg are  (1) no arrow points to a full line and no dashed line is  adjacent to a full line, (2) there is a valid ordering  of its  unique set of connected components,  $(g_1,\dots, g_Q)$.  These components  result  by removing all  arrows from \Greg\!.  The connected components of dashed lines  can be ordered such that  $g_q<g_l$ if there is an arrow pointing  to a node in $g_q$ from a node in $g_l$.  Connected components  of  full lines have a higher order than those of dashed lines and can be ordered in any way.

The induced subgraph of \Greg  of only arrows and dashed lines is called the  response set, $u$,   the complimentary one  of only full lines  is the context set, $v$ so that $N=(u,v).$  Accordingly, connected components of $u$ are the response components, those of  $v$ are the context components. 
For a node $i$ in  response set $g_q$,  its past is denoted by $\pst(i)$ and  it consists of  all nodes in components having a higher order than
 $g_q$.

Although for  the regression-graph structure, only  a  valid ordering of its connected components 
is essential, one can indeed extend a  valid ordering of these components into a complete valid ordering of all nodes, $(1,\dots,N)$, where $i<j$ if $q<l$ for $i\in g_q$ and   $j\in g_l$.  This implies that $i<j$ if there is an arrow pointing  to $i$ from $j$. 
The past of a node $i$  in  response component $g_q$  contains  never any node 
$j$ of $g_q$ even if $j$ is  larger in a complete valid ordering of the nodes.

Conditional independences are captured by  missing edges in \Greg\!.
A missing edge for $i,j$  in \Greg is denoted  by $i\nsim j$. For  $i\nsim j$, there is some subset $c$ of $N\setminus \{i,j\}$ such that in the graph $X_i$ is conditionally independent of $X_j$ given $X_c$, written
as $i\ci j|c$. Four different ways of defining $c$ are studied in the next Section 3.

\section{Pairwise Markov properties for regression graphs}
The following  properties have been defined  for conditional independences of  probability distributions. Let $a,b,d,c$ be disjoint subsets of $N$,  where $c$ may be the empty set.
\begin{itemize}	
\item[\n] \hspace{6mm} $(i)$ symmetry:  $a\ci b|c  {\implies}  b\ci a|c$\,;
\item[\n] \hspace{6mm} $(ii)$ decomposition:  $a\ci bd|c  {\implies}  (a\ci b|c$ {\em and} $a\ci d|c$)\,;
\item [\n] \hspace{6mm} $(iii)$ weak union:   $a\ci bd|c  {\implies} (a\ci b|dc$ {\em and} $a\ci d|bc$)\,;
\item [\n] \hspace{6mm} $(iv)$  contraction:   $( a \ci b|dc$  {\em and }   $a\ci d|c) \iff a \ci bd|c$\,;
\item [\n] \hspace{6mm} $(v)$  intersection:  ( $a \ci d|bc$  {\em and } $a\ci b|dc ) {\implies}  a\ci bd|c$\,;
\item [\n] \hspace{6mm} $(vi)$  composition:  ($a \ci d|c$ {\em and }$ a \ci b|c  ) {\implies}  a\ci bd|c$\,.
\end{itemize}
Note that composition defines the reverse direction of decomposition. Intersection gives the reverse
direction of  weak union and  properties $(i)$ to $(iv)$ are the independence properties of all probability distributions;  see \citet{Dawid79}, \citet{Stud89}.

The first five independence properties, listed above,  give the   graphoids of
 \citet{PearlPaz87},  see also \citet{Stud89, Stud05}.
If, in addition,  the composition property holds, one speaks of  a compositional graphoid.   For   compositional graphoids,  pairwise independences can be combined to obtain global independence statements  of the type $a\ci b\cd c$ for general classes of graphs, which include
regression graphs; see \citet{SadegLau14}.  Thus,   the above six properties are also satisfied by \Greg\!.  They are  the properties used  here in the proof of  Theorem \ref{thm:1}.

Let \Greg  be a regression graph with  a valid ordering of the node set $N=(1,\dots, d)$, which necessarily conforms with a valid ordering of its connected components. Based on this ordering,  we always assume in this section that  $i<j$ for the node pair $i,j$. To define  pairwise Markov properties for \Greg\!, we use the following notation for parents, anteriors and the past of node pair $i,j$:
\begin{eqnarray*}\pa(i,j)&=&\pa(i)\cup\pa(j)\setminus\{i,j\},\\ \ant(i,j)&=&\ant(i)\cup\ant(j)\setminus\{i,j\},\\ \pst(i,j)&=&\pst(i)\cup\pst(j)\setminus\{i,j\}.
\end{eqnarray*}

The distribution  $\mathcal{P}$ of $(X_n)_{n\in N}$ satisfies a pairwise Markov property  $(Pm)$, for $m=1,2,3,4$, with respect to \Greg if
for $i$ and $j$ in the context set, $v$,
\begin{eqnarray*}
i\not\sim j &\implies& i\ci j\cd v\setminus\{i,j\}\,,
\end{eqnarray*}
and for $i$ in the  response  set, $u$,
\begin{eqnarray*}
(P1):  i\not\sim j &\implies& i\ci j\cd \pst(i,j)\,,\\
(P2):  i\not\sim j &\implies& i\ci j\cd \ant(i,j)\,,\\
(P3): i\not\sim j& \implies & i\ci j\cd \pa(i,j)\,,\\
(P4): i\not\sim j& \implies &i\ci j\cd \pa(i).
\end{eqnarray*}
Notice that in (P4), $\pa(i)$ may be replaced by $\pa(j)$ whenever the two nodes are in the same connected component. 
Consider, for example,  the graph of Figure~\ref{fig:1} with $N=(1,\ldots, 9)$.
\begin{figure}[H]
\centering
\includegraphics[scale=0.55]{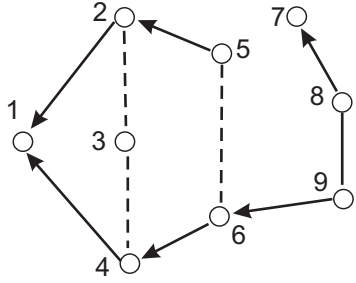}
\caption{{\footnotesize A regression graph.}}\label{fig:1}
\end{figure}

 For the uncoupled pair of nodes $2,4$,  the nodes in its past  are $\pst(2,4)=\{5,6,7,8, 9\}$, the anteriors are $\ant(2,4)=\{5,6,8,9\}$, the parents are $\pa(2,4)=\{5,6\}$ while the parent of node $2$ alone is node $5$, so that
by the four pairwise Markov properties, listed here including their conditioning sets,
\begin{eqnarray*}
  (P1), \text{past} & \implies & 2\ci 4|\{5,6,7, 8,9\};\\
  (P2), \text{anteriors} &\implies & 2\ci 4|\{5,6,8,9\};\\
  (P3), \text{joint parents} &  \implies & 2\ci 4|\{5,6\};\\
  (P4), \text{parents of the smaller node} & \implies & 2\ci 4|5.
\end{eqnarray*}

\begin{theorem}\label{thm:1}
Let a distribution $\mathcal{P}$ of $(X_n)_{n\in N}$  be a compositional graphoid. Then, for every valid ordering of the nodes of \Greg\!, $\mathcal{P}$ satisfies $(P1)\iff (P2)\iff (P3)\iff (P4)$ with respect to \Greg\!.
\end{theorem}

Before presenting the proof, we provide the following example related to the graph of Figure~\ref{fig:1}, in which we follow
the same method as in the proof of this theorem that is using compositional graphoid properties  to obtain the equivalence of the independence statements implied by the different pairwise Markov properties:

{\bf $(P1)\Rightarrow(P2):$}  By definition,   we have $2\ci 4\cd \{5,6,7, 8,9\}$ and $4\ci 7\cd \{5,6,8,9\}$. Contraction implies $4\ci\{2,7\}\cd \{5,6,8,9\}$. Then, decomposition  gives that $2\ci 4\cd \{5,6,8,9\}$.

{\bf $(P2)\Rightarrow(P3):$} Nodes $6, 8$ and $9$ are in the set $\ant(2,4)\setminus\pa(2)$, and it holds that  $2\ci 8\cd \{5,6,9\}$, $2\ci 9\cd \{5, 6,8\}$, and
$2\ci 6 \cd \{5,8,9\}$. Intersection,  used repeatedly, implies $2\ci\{6, 8,9\}\cd 5$.
 This together with $2\ci 4\cd \{5,6,8,9\}$  implies by  contraction $2\ci\{4,6,8,9\}\cd 5$. By decomposition, $2\ci\{4,6\}\cd 5$ so that  by weak union,   $2\ci 4\cd \{5,6\}$.

{\bf $(P3)\Rightarrow(P4):$} By contraction,  $2\ci 9\cd 5$ and $2\ci 6\cd \{9,5\}$,  imply $2\ci \{6,9\}\cd 5$, which by decomposition implies $2\ci 6\cd 5$.  By contraction,  this and  $2\ci 4\cd \{5,6\}$ imply $2\ci\{4,6\}\cd 5$ so that by  decomposition $2\ci 4\cd 5$.

{\bf $(P4)\Rightarrow(P1):$} Composition used repeatedly for statements $2\ci 4\cd 5$, $2\ci 6 \cd 5$,  $2\ci 7\cd 5$, $2\ci 8\cd 5$, $2\ci 9\cd 5$ gives $2\ci\{4,6,7,8,9\}\cd 5$. Weak union implies $2\ci 4\cd \{5,6,7,8,9\}$.

\begin{proof}[Proof of Theorem \ref{thm:1}]
{\bf $(P1)\Rightarrow(P2):$}  By definition, $\pst(i,j)=\pst(i)\setminus\{j\}$. Let $k\in \{j\}\cup\pst(i)\setminus \ant(i,j)$, then $k\nsim i$ and  $i\ci k\cd \pst(i)\setminus\{k\}$. For another possible node $l \in\pst(i)\setminus \ant(i,j)$ holds also that $i\ci l\cd \pst(i)\setminus\{l\}$. By the intersection property, we obtain $i\ci\{k,l\}\cd \pst(i)\setminus\{k,l\}$. By the same method and an inductive argument, for all members of $\pst(i)\setminus \ant(i,j)$, we obtain $i\ci\pst(i)\setminus\ant(i,j)\cd \ant(i,j)$. Now if $i$ and $j$ are in different connected components then by the same method, we use the intersection property to obtain $i \ci\{j\}\cup\pst(i)\setminus \ant(i,j)\cd \ant(i,j)$, and if $i$ and $j$ are in the same connected component, we use the contraction property with $i\ci j\cd \pst(i)$ to obtain the same independence statement. The decomposition property yields $i\ci j\cd \ant(i,j)$.

{\bf $(P2)\Rightarrow(P3):$} Suppose that $i$ is in the lower or equal index component $g_r$ than that of $j$. Let also $k\in\ant(i,j)\setminus\pa(i)$. Notice that $k\nsim i$ since $k\not\in\pa(i)$ and if, $k\in\ant(j)\setminus\ant(i)$ and adjacent to $i$ then $i$ is of a higher order than $j$, which is not possible. We prove that $i\ci k\cd \ant(i)\setminus\{k\}$:

If $k\in\ant(i)$ then $\ant(k)\subset\ant(i)$ and subsequently $\ant(i,k)=\ant(i)$. The result is then obvious from the assumption. Thus assume that $k\in\ant(j)\setminus\ant(i)$. We then use   reverse induction on the order of the connected components in which $k$ lies to prove the claim. For the base, we have that $k$ is in $g_v$, and let $g_{v,k}$ be the connected component containing $k$ in $g_v$. It holds that $\ant(k)=g_{v,k}$. For every member $l$ of $g_{v,k}$, it holds that $i\ci l\cd \ant(i)\cup g_{v,k}$. By intersection, for all such statements, we obtain $i\ci g_{v,k}\cd \ant(i)$. Decomposition implies the result.

Now suppose that $k\in g_q$ and, for every $l\in\ant(k)\setminus\ant(i)$, by the induction hypothesis, it holds that $i\ci l\cd \ant(i)$. By the composition property for all such $i\ci l\cd \ant(i)$, we obtain that $i\ci\ant(k)\setminus\ant(i)\cd \ant(i)$. This together with $i\ci k\cd \ant(i,k)$, by using the contraction property, implies that $i\ci\{k\}\cup\ant(k)\setminus\ant(i)\cd \ant(i)\setminus \{k\}$. The decomposition property gives $i\ci k\cd \ant(i)\setminus \{k\}$.

Now by the intersection property for all such $i\ci k\cd \ant(i)\setminus\{k\}$, we obtain $i\ci \ant(i,j)\setminus\pa(i)\cd\pa(i)$. This together with $i\ci j\cd\ant(i,j)$, using the contraction property, implies $i\ci\{j\}\cup\ant(i,j)\setminus\pa(i)\cd\pa(i)$. By the decomposition property, this  gives $i\ci\{j\}\cup\pa(j)\setminus\pa(i)\cd\pa(i)$. The weak union property now implies $i\ci j\cd\pa(i,j)$.

{\bf $(P3)\Rightarrow(P4):$} We prove the result by reverse induction on the order of the connected component that contains $j$. If $j$ is in $g_v$ then the result trivially holds. Thus suppose that $j$ lies in $g_q$. By induction hypothesis for each $k\in\pa(j)\setminus\pa(i)$ we have $i\ci k\cd\pa(i)$. By the composition property, it is implied that  $i\ci\pa(j)\setminus\pa(i)\cd\pa(i)$. This together with $i\ci j\cd\pa(i,j)$, by the contraction property, implies $i\ci\{j\}\cup\pa(j)\setminus\pa(i)\cd\pa(i)$. The decomposition property now gives $i\ci j\cd\pa(i)$.

{\bf $(P4)\Rightarrow(P1):$} Suppose that $k\in\{j\}\cup\pst(i,j)\setminus\pa(i)$. We have that $k\nsim i$. Therefore, $i\ci k\cd\pa(i)$. Using the composition property for all such statements for members of $\{j\}\cup\pst(i,j)\setminus\pa(i)$, we obtain $i\ci\{j\}\cup\pst(i,j)\setminus\pa(i)\cd\pa(i)$. The weak union property implies $i\ci j\cd \pst(i.j)$.
\end{proof}
Based on the equivalence of the global Markov property and the pairwise Markov property $(P2)$, proven in \citet{SadegLau14}, we have the following corollary:
\begin{coro}\label{cor:1}
Let a distribution $\mathcal{P}$  of $(X_n)_{n\in N}$ be a compositional graphoid. Then, for every valid ordering of the nodes of \Greg, $\mathcal{P}$ satisfies the global Markov property with respect to \Greg if and only if it satisfies any one of $(P1)$, $(P2)$, $(P3)$, or $(P4)$.\end{coro}

It is a consequence of  the above results that in case one of the four pairwise
independence properties is satisfied,  all other three  hold as well so
that, depending on the purpose of an enquiry, each one of them may be used. The results have been applied implicitly in  model-fitting strategies which also include checks
of necessary conditions for properties  $(v)$ to $(vii)$; see e.g.~\citet{WerSadeg12}, Appendix;
\citet{WerCox13}, Appendix. 

There is an additional independence property, needed
for studying dependences induced by  a given edge-minimal graph,
\begin{itemize}	
\item [\n]  \hspace{6mm} $(vii)$ singleton transitivity: ($j\ci k|c$ {\em and } $j\ci k|ic) {\implies}  (i\ci j|c$ {\em or } $i\ci k|c$),
\end{itemize}	
where $i,j,k$ are distinct nodes of node set $N$ and $c\subseteq N\setminus\{i,j,k\}$.
 Probability distributions satisfying this property of singleton transitivity in addition to $(i)$ to $(vi)$  are  for instance regular
joint Gaussian distributions; see \citet{LlenMat07}, and totally positive distributions with positive support everywhere;  see \citet{FallatEtal16}.

The above seven properties of a density generated over \Greg ensure that
the graph may be used to trace
 pathways of dependences, see \citet{Wer12,Wer15}. Such a sequence of
regressions has therefore been said to be  traceable.

\section{Discussion}
For regression graphs and for distributions that are compositional graphoids, the equivalence of four different  pairwise independence properties has been proven,   two of which, $(P1)$ and $(P4)$, require a valid ordering of all variables.   Together with the known equivalence of one of them, $(P2)$, to the global Markov property, this gives essential  insights  into  regression-graph structures and, at the same time,  into  the structural independence constraints of such distributions that is  into   independences which hold irrespective of specific  parameter values for a given regression graph model. 

For the same set of context variables, different special parametrizations for the directed dependences of main interest may arise,  for instance,  just with differing baseline conditions that is with different realisations for the context variables. Hence, especially for comparing results from several  studies with the same variables, it is essential
to be able to distinguish between structural independences and those which result only  with specific constellations of parameters. The former are captured
by a generating regression graph and the latter by parameter estimates  given a set of data.

The conditioning set is largest for the past of  any two non-adjacent nodes, $(P1)$,
and smallest for the subset of this past containing just the parents of one of the two nodes, $(P4)$. The two  other properties have  conditioning sets  of intermediate size, they may  be larger than those for $(P4)$ but they are never larger than those  for $(P1)$. This implies for a  regression graph what follows
for the generated distribution by its factorization in terms of connected components: variables outside the past of a missing edge 
need not be considered for its independence interpretation.

For any given joint response, property $(P4)$ justifies the elimination of unimportant regressors for each response component separately, but only with property $(P3)$, that is by  conditioning on the joint parents of two response nodes, possible estimation problems  are avoided, such as in seemingly unrelated regressions. By using $(P3)$, a reduced model is replaced by a  covering model which is simpler concerning estimation properties; see \citet{CoxWer90}. In the Gaussian case for a reduced model of seemingly unrelated regressions, the covering model is 
the so-called general linear model which has identical sets of regressors for each response pair; see \citet{Anderson58}.

For tracing pathways of dependence, the pairwise property given the anterior set of a node, $(P2)$ is the essential one. Intervening early at an inner node of an
anterior path  in an edge-minimal graph will interrupt a pathway of development.  For instance with  exclusively strong risk factors  along such a  path, an early intervention may stop an otherwise  disastrous  accumulation of risks.

The main  importance of the pairwise property given the past of a node pair, $(P1)$, is that independence constraints and pairwise dependences are defined  without any change in this largest conditioning set. This permits to extend a given connected graph into a complete graph by keeping  its given valid node ordering: all missing full lines in the context set are added, in each response component dashed lines are added until its subgraph is complete,  then for each remaining node pair with $i<j$, an arrow pointing to $i$ from $j$ is added. This helps to understand which constraints have been imposed on this completed graph to obtain the given missing edges.

The completed graph leads also to the  factorization of a  generated distribution in terms of connected components  but without any conditional independence constraints. Such a covering model is sometimes called  a saturated regression graph model. This saturated model is the appropriate basis  for discussions with researchers in the field on which alternative orderings may possibly  be more in line with the available knowledge about the variables under study.  A disadvantage of $P1$ is for efficient testing: it may include many variables as mere noise; see \citet{Altham84}.

For general types of  distributions generated over a regression graph,  efficient algorithms still need to be developed  to decide whether they satisfy  the independence properties $(v)$ to $(vii)$. These are not common to all probability distributions but hold, for instance, in joint Gaussian distributions and in totally positive distributions with positive support everywhere. Joint Bernoulli distributions satisfying other types of conditions will in a future paper be shown to also satisfy the three additional properties.  In another future paper,  additional conditions will  be given under which  these
distributions contain precisely those independences  provided as  structural by an edge-minimal regression graph.

\subsection*{Acknowledgments}
Work of the first author was partially supported by grant $\#$FA9550-14-1-0141 from the U.S. Air Force Office of Scientific
Research (AFOSR) and the Defense Advanced Research Projects Agency (DARPA). The second author thanks GM Marchetti and M Mouchart for stimulating discussions and  comments.

 \renewcommand{\baselinestretch}{1.1}
\renewcommand\refname

\end{document}

%% file: deff.tex
\usepackage[a4paper,left=28mm,right=25mm,top=30mm,bottom=30mm]{geometry}
\usepackage[small]{caption} 
\RequirePackage[OT1]{fontenc}

\usepackage{dcolumn,warpcol}  
\RequirePackage{amsthm,amsmath,amssymb}
\usepackage{graphicx}\graphicspath{{pics/}}

\usepackage{latexsym}
\usepackage{arydshln,booktabs}
\usepackage{array}

\usepackage{here}
\usepackage{setspace}
\usepackage{bm} 
\usepackage{lscape} 

\usepackage{bm} 
\usepackage[latin1]{inputenc}
\usepackage{color}
\usepackage{graphicx}
\RequirePackage[OT1]{fontenc}
\usepackage{latexsym}
\usepackage{booktabs,rccol,flafter,arydshln,warpcol}
\RequirePackage{amsthm,amsmath}

\usepackage{here}

\newcommand{\Greg}{$G^N_{\mathrm{reg}}\,$}

\newcommand{\cd}{|}

\newcommand{\pa}{\mathrm{par}}
\newcommand{\pst}{\mathrm{pst}}

\newcommand{\ant}{\mathrm{ant}}

\newtheorem{coro}{Corollary}

\newtheorem{theorem}{Theorem}

\newcommand{\ci}{\mbox{\protect{ $ \perp \hspace{-2.3ex}
\perp$ }}}

\newcommand{\n}[0]{\hspace*{.35em}}

\newcommand{\nn}[0]{\hspace*{.7em}}

\newcommand{\snode}{\mbox {\large
{$\mbox{$\circ$}$}}}

\newcommand{\ful}{\mbox{$\, \frac{ \nn \nn \;}{ \nn \nn
}$}}

\newcommand{\fla}{\mbox{$\hspace{.05em} \prec
\!\!\!\!\!\frac{\nn \nn}{\nn}$}}

\newcommand{\dal}{\mbox{$  \frac{\n}{\n}
\frac{\; \,}{\;}  \frac{\n}{\n}$}}

\makeatletter
\renewcommand\section{\@startsection{section}{1}{\z@}%
{-3.25ex\@plus -1ex \@minus -.2ex}{1.5ex \@plus .2ex}%
{\normalfont\large\bfseries}}
\renewcommand\subsection{\@startsection{subsection}{2}{\z@}%
{-3.25ex\@plus -1ex \@minus -.2ex}%
{1.5ex \@plus .2ex}%
{\normalfont\normalsize\bfseries}}
\renewcommand\subsubsection{\@startsection{subsubsection}{3}{\z@}%
{-3.25ex\@plus -1ex \@minus -.2ex}%
{1.5ex \@plus .2ex}%
{\normalfont\normalsize\bfseries}}
\renewcommand\paragraph{\@startsection{paragraph}{4}{\parindent}%
{3.25ex \@plus1ex \@minus .2ex}%
{-1em}%
{\normalfont\normalsize\bfseries}}
\makeatother